\documentclass[11pt]{article}
\usepackage{a4wide}

\usepackage[a4paper, margin=2.3cm]{geometry}
\usepackage{amsmath,amssymb,amsthm}
\usepackage{color}
\usepackage{parskip}
\usepackage{hyperref}
\usepackage{parskip}
\usepackage{setspace}
\usepackage{graphicx}
\usepackage{mathtools}
\usepackage[thinc]{esdiff}
 
\newtheorem{theorem}{Theorem}[section]
\newtheorem{remark}{Remark}[section]
\newtheorem{corollary}[theorem]{Corollary}
\newtheorem{proposition}[theorem]{Proposition}
\newtheorem{lemma}[theorem]{Lemma}

\newtheorem{definition}[theorem]{Definition}

\newtheorem*{definition*}{Definition}

\begin{document}
\title{An inverse-type problem for cycles in local Cayley distance graphs}
\author{Thang Pham\thanks{Theory of Combinatorial Algorithms Group, ETH Zurich, Switzerland. Email: phamanhthang.vnu@gmail.com}}
\date{}
\maketitle
\begin{abstract}
Let $E$ be a proper symmetric subset of $S^{d-1}$, and $C_{\mathbb{F}_q^d}(E)$ be the Cayley graph with the vertex set $\mathbb{F}_q^d$, and two vertices $x$ and $y$ are connected by an edge if $x-y\in E$. Let $k\ge 2$ be a positive integer.  We show that for any $\alpha\in (0, 1)$, there exists $q(\alpha, k)$ large enough such that if $E\subset S^{d-1}\subset \mathbb{F}_q^d$ with $|E|\ge \alpha q^{d-1}$ and $q\ge q(\alpha, k)$, then for each vertex $v$, there are at least $c(\alpha, k)q^{\frac{(2k-1)d-4k}{2}}$ cycles of length $2k$ with distinct vertices in $C_{\mathbb{F}_q^d}(E)$ containing $v$. This result is the  inverse version of a recent result due to Iosevich, Jardine, and McDonald (2021). 
\end{abstract}
\noindent \textbf{Keywords}: Spherical configurations, Finite fields, Cayley graphs, Cycles. \\ \textbf{Mathematics Subject Classification}:  52C10, 11T23\\
\section{Introduction}
Let $G$ be an abelian finite group and a symmetric set $E\subset G$. The Cayley graph $C_G(E)$ is defined as the graph with the vertex set $V=G$, and there is an edge from $x$ to $y$ if $y-x\in E$. Let $\mathbb{F}_q$ be a finite field of order $q$, where $q$ is a prime power. In this paper, we consider $G$ being the whole vector space $\mathbb{F}_q^d$.

We have $C_{\mathbb{F}_q^d}(E)$ is a regular graph of degree $|E|$ with $q^d$ vertices. It is well--known in the literature that eigenvalues of $C_{\mathbb{F}_q^d}(E)$ are of the form  $\lambda_m:=\sum_{x\in E}\chi(x\cdot m)=\widehat{E}(m), ~m\in \mathbb{F}_q^d,$
where $\chi$ is the principle additive character of $\mathbb{F}_q$. Define $\mu:=\max_{m\ne (0, 0,\ldots, 0)} |\lambda_m|$. This quantity is referred as the second largest eigenvalue of $C_{\mathbb{F}_q^d}(E)$. We call a graph $(n, d, \lambda)$-graph if it has $n$ vertices, the degree of each vertex is $d$, and the second largest eigenvalue is at most $\lambda$.

When $E=S^{d-1}$, the unit sphere in $\mathbb{F}_q^d$, we recall a result from a paper of Iosevich and Rudnev \cite[Lemma 5.1]{ir} that $\mu=(1+o(1))q^{\frac{d-1}{2}}$. Thus, the graph $C_{\mathbb{F}_q^d}(S^{d-1})$ is a $(n, d, \lambda)$-graph with $n=q^d$, $d=|S^{d-1}|$, and $\lambda=(1+o(1))q^{\frac{d-1}{2}}$. In a $(n, d, \lambda)$-graph, we know from \cite[Theorem 4.10]{expander} that any large subset of vertices contains the correct number of copies of any fixed sparse graph. More precisely, let $H$ be a fixed graph with $r$ edges, $s$ vertices, and maximal degree $\Delta$, then any subset $A\subset \mathbb{F}_q^d$ of $m$ vertices with $m\gg \lambda\left(\frac{n}{d}\right)^\Delta$ contains about $m^s(d/n)^r$ copies of $H$. The condition $m\gg \lambda\left(\frac{n}{d}\right)^\Delta$ can be improved when $H$ is of some specific configuration, for instance, paths, stars, complete graphs \cite{bene, chapman, 11, 17, io2021, 34, vinh}. \footnote{We use the following notations: $X \ll Y$ means that there exists  some absolute constant $C>0$ such that $X \leq CY$, $X\sim Y$ means that $X\ll Y\ll X$, $X=o(Y)$ means that $\lim_{q\to \infty}X/Y=0$.}

In the inverse setting, we let $E$ be a proper subset of $S^{d-1}$ and $A=\mathbb{F}_q^d$, the question is to find conditions on $E$ such that the graph $C_{\mathbb{F}_q^d}(E)$, which will be called \textit{local Cayley distance graphs}, contains at least one copy of $H$. 

The main purpose of this paper is to study the inverse version of a recent result due to Iosevich, Jardine, and McDonald in \cite{io2021} on the distribution of cycles. We start by stating their result. 
\begin{theorem}[Iosevich-Jardine-McDonald, \cite{io2021}]\label{thm11}
Let $A$ be a set in $\mathbb{F}_q^d$. Suppose that $|A|\gg q^{\frac{d+2}{2}}$, then for any positive integer $\ell\ge 3$, the number of cycles of length $\ell$ in $C_{\mathbb{F}_q^d}(S^{d-1})$ with vertices in $A$ is  $(1+o(1))|A|^{\ell}q^{-\ell}$. In addition, when $\ell$ is large, then the exponent $\frac{d+2}{2}$ can be improved, namely, the condition 
\[|A|\ge \begin{cases}q^{\frac{1}{2}\left(d+2-\frac{\ell-4}{\ell-2}+\delta \right)}, ~\mathtt{if}~\ell\ge 4~ \mathtt{even}\\ q^{\frac{1}{2}\left(d+2-\frac{\ell-3}{\ell-1}+\delta \right)}, ~\mathtt{if}~\ell\ge 3~ \mathtt{odd}\end{cases},\]
where $0<\delta\ll \frac{1}{\ell^2}$,
would be enough. 
\end{theorem}
The following is our main result. 
\begin{theorem}\label{thm-lowerbound}
Let $d, k\in \mathbb{N}$ with $d\ge 4k+2$, $\alpha\in (0, 1)$ and $q\ge q(\alpha, k)$. For a symmetric set $E\subset S^{d-1}\subset \mathbb{F}_q^d$ with $|E|\ge \alpha q^{d-1}$,  the number of cycles of length $2k$ in $C_{\mathbb{F}_q^d}(E)$ with distinct vertices passing through each vertex of $C_{\mathbb{F}_q^d}(E)$ is at least $c(\alpha, k)q^{\frac{(2k-1)d-4k}{2}}$.
\end{theorem}

To prove Theorem \ref{thm-lowerbound}, several serious challenges arise, and the most difficulty comes from the fact the graph $C_{\mathbb{F}_q^d}(E)$ is not a \textit{pseudo-random graph}, namely, the second eigenvalue $\mu$  is arbitrary close to the graph degree when $q$ is large enough. 

\begin{proposition}\label{thm1.9}
For any $1\le m\ll q^{d-1}$ and $\epsilon>0$ with $1/\epsilon\in \mathbb{Z}$. Let $q=p^{\frac{1}{\epsilon}}$. There exists $E\subset S^{d-1}$ such that $|E|=m$ and $\mu \ge \frac{|E|}{2q^{\epsilon}}$.  In addition, if $|E+E|\sim |E|$, then we have $\mu\sim \lambda_{(0, \ldots, 0)}= |E|$. 
\end{proposition}

Hence, it is not possible to apply techniques of pseudo-random graphs to prove such a result as Theorem \ref{thm-lowerbound}. Our main ingredient is a recent Ramsey-type result on the number of congruence copies of $2k$-\textit{spherical configurations} spanning $2k-2$ dimensions due to Lyall, Magyar, and Parshall in \cite{lyall}, which has been derived by  using a \textit{generalized von-Neumann type inequality} \cite[Proposition 6]{lyall} and \textit{an inverse theorem} \cite[Proposition 7]{lyall}.

It seems difficult to extend the approach of Theorem \ref{thm-lowerbound} for other subgraphs $H$. When $H$ is a $k$-simplex, say $k=2$ for simplicity, the inverse problem asks for conditions on three given proper subsets $E_1, E_2, E_3$ of $S^{d-1}$ such that there are three vertices $x, y, z\in \mathbb{F}_q^d$ such that $x-y\in E_1, y-z\in E_2, z-x\in E_3$. Note that $E_1, E_2, E_3$ can also be assumed to be subsets of spheres with different radii. We believe that finding a non-trivial solution of this problem would be much difficult compared to the original one.

When $E=S^{d-1}$, giving a lower bound on the number of cycles  in $C_{\mathbb{F}_q^d}(E)$ is much easier, since, as mentioned earlier, $C_{\mathbb{F}_q^d}(S^{d-1})$ is a pseudo-random graph with the second eigenvalue $\mu\sim \sqrt{|S^{d-1}|}$. In the next proposition, we provide an improvement of Theorem \ref{thm11} in terms of the lower bound on the number of cycles of even length.

\begin{proposition}\label{pro14} Suppose $E=S^{d-1}$, then the number of cycles of length $2k$ in $C_{\mathbb{F}_q^d}(S^{d-1})$ is $(1+o(1))|S^{d-1}|^{2k-1}q^{d-1}$. In addition, for any set $A\subset \mathbb{F}_q^d$ with $|A|\gg \min \{q^{\frac{d+1}{2}}, q^\frac{k}{k-1}\}$, the number of cycles of length $2k$ in $C_{\mathbb{F}_q^d}(E)$ with vertices in $A$ is at least $q^{-2k}|A|^{2k}$.
\end{proposition}
Based on Proposition \ref{pro14} and in the spirit of Theorem \ref{thm11}, we conjecture that for any set $A\subset \mathbb{F}_q^d$ with $|A|\gg \min \{q^{\frac{d+1}{2}}, q^\frac{k}{k-1}\}$, the number of cycles of length $2k$ in $C_{\mathbb{F}_q^d}(S^{d-1})$ with vertices in $A$ is equal to  $(1+o(1))q^{-2k}|A|^{2k}$.
\section{Preliminaries}
%Before stating our main results, we need to introduce some notations.
Let $\chi\colon \mathbb{F}_q\to \mathbb{S}^1$ be the canonical additive character. For example, if $q$ is a prime number, then $\chi(t)=e^{\frac{2\pi i t}{q}}$, if $q=p^n$, then we set $\chi(t)=e^{\frac{2\pi i \mathtt{Tr}(t)}{q}}$, where $\mathtt{Tr}\colon \mathbb{F}_q\to \mathbb{F}_q$  is the trace function defined by $\mathtt{Tr}(x):=x+x^p+\cdots+x^{p^{n-1}}$. 

We recall the orthogonal property of $\chi$: for any $x\in \mathbb{F}_q^d$, $d\ge 1$, 
\[\sum_{m\in \mathbb{F}_q^d}\chi(x\cdot m)=\begin{cases}0~~\mathtt{if}~x\ne (0, \ldots, 0) \\ q^d ~\mathtt{if}~x=(0, \ldots, 0)\end{cases},\]
where $x\cdot m=x_1m_1+\cdots+x_dm_d$. 

For any $x\in \mathbb{F}_q^d$, through this paper, we define $||x||=x_1^2+\cdots+x_d^2$. 

Given a set $E\subset \mathbb{F}_q^d$, we identify $E$ with its indicator function $1_E$. The Fourier transform of $E$ is defined by 
\[\widehat{E}(m):=\sum_{x\in \mathbb{F}_q^d}E(x)\chi(-x\cdot m).\]
Let $E$ be a set in $\mathbb{F}_q^d$, and $k$ be a positive integer. The $k$--additive energy of $E$, denoted by $T_k(E)$, is defined by
\[T_k(E):=\#\left\lbrace (a_1, \ldots, a_k, b_1\ldots, b_k)\in E^{2k}\colon a_1+\cdots+a_k=b_1+\cdots+b_k\right\rbrace.\]
We call such a tuple $(a_1, \ldots, a_k, b_1, \ldots, b_k)$ \textit{$k$-energy tuple}.
 
A $k$-energy tuple $(a_1, \ldots, a_k, b_1, \ldots, b_k)\in \left(\mathbb{F}_q^d\right)^{2k}$ is called \textit{good} if for any two sets of indices $I, J\subset \{1, \ldots, k\}$, we have $\sum_{i\in I}a_i-\sum_{j\in J}b_j\ne 0$. We denote the number of good $k$-energy tuples with vertices in $E$ by $T_k^{\mathtt{good}}(E)$.
 
In the next lemma, we show that for every vertex $v\in \mathbb{F}_q^d$, the number of cycles of length $2k$ with distinct vertices going through $v$ is at least $T_k^{\mathtt{good}}(E)$.
\begin{lemma}\label{con}
For any $k\ge 2$ and any $v\in \mathbb{F}_q^d$, the number of cycles of length $2k$ in $C_{\mathbb{F}_q^d}(E)$ with distinct vertices going through $v$ is at least $T_k^{\mathtt{good}}(E)$.
\end{lemma}
\begin{proof}
For each good $k$-energy tuple $(a_1, \ldots, a_k, b_1, \ldots, b_k)\in E^{2k}$, we consider the following cycle of length $2k$ in $C_{\mathbb{F}_q^d}(E):$
\[v, v+a_1, v+a_1+a_2, \ldots, v+a_1+\cdots+a_k, v+\sum_{i=1}^ka_i-b_1, \cdots, v+\sum_{i=1}^ka_i-\sum_{i=1}^{k-1}b_i.\]
We observe that in this cycle, each vertex appears only one time since the $k$-energy tuple is good. So, for each vertex $v$, there are at least $T_k^{\mathtt{good}}(E)$ cycles with distinct vertices passing through $v$. 
%Let $M$ be its adjacency matrix of the graph $C_{\mathbb{F}_q^d}(E)$. As we mentioned in the introduction that eigenvalues of $C_{\mathbb{F}_q^d}(E)$ are of the form 
%\[\lambda_m:=\sum_{x\in E}\chi(x\cdot m)=\widehat{E}(m), ~m\in \mathbb{F}_q^d.\]
%We know that the number of cycles of length $2k$ in $C_{\mathbb{F}_q^d}(E)$ is equal to the trace of the matrix $M^{2k}$, which is 
%\[\sum_{m\in \mathbb{F}_q^d}\lambda_m^{2k}=\sum_{m\in \mathbb{F}_q^d}|\widehat{E}(m)|^{2k}.\]
%By the orthogonality of $\chi$, we have 
%\begin{align*}\label{connection}
%\sum_{m\in \mathbb{F}_q^d} |\widehat{E}(m)|^{2k}&=q^d\cdot \#\{(a_1, \ldots, a_k, b_1, \ldots, b_k)\in E^{2k}\colon a_1+\cdots+a_{k}=b_1+\cdots+b_k\}\nonumber\\
%&=q^d\cdot T_k(E).
%\end{align*}
\end{proof}
We also recall the well--known Expanding mixing lemma for regular graphs. We refer the reader to \cite{hanson, expander} for proofs.
\begin{lemma}\label{exp} Let $\mathcal{G}$ be a regular graph with $n$ vertices of degree $d$. Suppose that the second eigenvalue of $\mathcal{G}$ is at most $\mu$, then for any two vertex sets $U$ and $W$ in $\mathcal{G}$, the number of edges between $U$ and $W$, denoted by $e(U, W)$, satisfies 
\[\left\vert e(U, W)-\frac{d|U||W|}{n}\right\vert\le \mu |U|^{1/2}|W|^{1/2}.\] 
When $U$ and $W$ are multi-sets, we also have 
\[\left\vert e(U, W)-\frac{d|U||W|}{n}\right\vert\le \mu \left(\sum_{u\in \overline{U}}m(u)^2\right)^{1/2}\cdot \left(\sum_{w\in \overline{W}}m(w)^2\right)^{1/2},\]
where $\overline{X}$ is the set of distinct elements in $X$, and $m(x)$ is the multiplicity of $x$.  
\end{lemma}
\section{Proof of Theorem \ref{thm-lowerbound}}
Theorem \ref{thm-lowerbound} follows directly from Lemma \ref{con} and the following lower bound for $T_k^{\mathtt{good}}(E)$. 
\begin{theorem}\label{kenergy}
Suppose $E$ satisfies assumptions of Theorem \ref{thm-lowerbound}, we have 
\[T_k^{\mathtt{good}}(E)\ge c(\alpha, k)q^{\frac{(2k-1)d-4k}{2}}.\]
\end{theorem}
In the rest of this section, we focus on proving Theorem \ref{kenergy}.

For each $j\ne 0$, let $S_j^{d-1}(x)$ be the sphere centered at $x\in \mathbb{F}_q^d$ of radius $j$. For the sake of simplicity, we write $S_j^{d-1}$ for $S_j^{d-1}(0, \ldots, 0)$, and $S^{d-1}$ for $S_1^{d-1}(0, \ldots, 0)$.
 
\begin{definition}
Let $X\subset \mathbb{F}_q^d$ be a configuration. We say that $X$ is spherical if $X\subset S_1^{d-1}(x)$ for some $x\in \mathbb{F}_q^d$. If $\dim (\mathtt{Span}(X-X))=k$, then we say $X$ spans $k$ dimensions. 
\end{definition}
The following result is our key ingredient in the proof of Theorem \ref{kenergy}.
\begin{theorem}[Lyall-Magyar-Parshall, \cite{lyall}]\label{lyalll}
Let $d, k\in \mathbb{N}$ with $d\ge 2k+6$, $\alpha\in (0, 1)$ and $q\ge q(\alpha, k)$. For $E\subset S^{d-1}$ with $|E|\ge \alpha q^{d-1}$, then $E$ contains at least $c(\alpha, k) q^{\frac{(k+1)d-(k+1)(k+2)}{2}}$ isometric copies of every non-degenerate $(k+2)$-point spherical configuration spanning $k$ dimensions.
\end{theorem}

This theorem says that for any $\alpha\in (0, 1)$ and any fixed non-degenerate $(k+2)$-point spherical configuration $X$ spanning $k$ dimensions, there exists $q_0=q_0(\alpha, k)$ which is large enough, such that for any $E\subset S^{d-1}\subset \mathbb{F}_q^d$ with $|E|\ge \alpha q^{d-1}$ and $q\ge q_0$, $E$ contains many isometric copies of $X$.  More precisely, let
\[X=\{\mathbf{0}, v_1, \ldots, v_k, a_1v_1+\cdots+a_kv_k\},\]
where $\mathbf{0}=(0, \ldots, 0), ~v_1, \ldots, v_k\in \mathbb{F}_q^d$ are linearly independent vectors, and $a_1, \ldots, a_k\in \mathbb{F}_q$, be a non-degenerate spherical configuration of $k+2$ points in $\mathbb{F}_q^d$ that spans a $k$-dimensional vector space. By non-degenerate, we meant that $\{\mathbf{0}, v_1, \ldots, v_k\}$ form a $k$--simplex with all non-zero side-lengths.  Assume that $E\subset S^{d-1}$ satisfying the conditions of Theorem \ref{lyalll}, then $E$ contains at least $c(\alpha, k)q^{\frac{(k+1)d-(k+1)(k+2)}{2}}$ copies of $X$ of the form 
\[X'=\{x_0, x_0+x_1, \ldots, x_0+x_k, x_0+a_1x_1+\cdots+a_kx_k\},\]
with $x_1, \ldots, x_k$ linearly independent such that $x_i\cdot x_j=v_i\cdot v_j$ for $1\le i\le j\le k$.

We recall that two configurations $X$ and $X'$ in $S^{d-1}$ are said to be in the same congruence class if there exists $g\in O(d, \mathbb{F}_q)$, the orthogonal group in $\mathbb{F}_q^d$, such that $g(X)=X'$. 

Let $Q$ be the set of distinct congruence classes of spherical configurations $X$ of the form 
\[X=\{x_0, x_0+x_1, x_0+x_2, \ldots, x_0+x_{2k-2}, x_0+\sum_{i=1}^{2k-2}(-1)^{i+1}(x+x_i)\},\]
satisfying
\begin{itemize}
\item $\{x_1, \ldots, x_{2k-2}\}$ are linearly independent. 
\item $||x_i-x_j||\ne 0$, $||x_i||\ne 0$ for all $1\le i\ne  j\le 2k-2$.
\item $X$ forms a good $k$-energy tuple. 
\end{itemize}
We note that vectors in $X\in Q$ form a $k$-energy tuple since 
\[x_0+(x_0+x_1)+(x_0+x_3)+\cdots+(x_0+x_{2k-3})=(x_0+x_2)+(x_0+x_4)+\cdots+(x_0+x_{2k-2})+u,\]
where $u=x_0+\sum_{i=1}^{2k-2}(-1)^{i+1}(x_0+x_i)$. 

For each $X\in Q$, let $N(X)$ be the number of congruent copies of $X$ in $E$. Set $N(Q)=\sum_{X\in Q}N(X)$. The next lemma gives us a lower bound for $T_k^{\mathtt{good}}(E)$.
\begin{lemma}\label{step1}
Suppose $E$ satisfies assumptions of Theorem \ref{thm-lowerbound}, we have 
\begin{equation}\label{eq10*}T_k^{\mathtt{good}}(E)\ge N(Q).\end{equation}
\end{lemma}
\begin{proof}
Let \[X=\{x_0, x_0+x_1, x_0+x_2, \ldots, x_0+x_{2k-2}, x_0+\sum_{i=1}^{2k-2}(-1)^{i+1}(x_0+x_i)\}\in Q,\]
and set $u=x_0+\sum_{i=1}^{2k-2}(-1)^{i+1}(x_0+x_i)$, then we have 
\[x_0+(x_0+x_1)+(x_0+x_3)+\cdots+(x_0+x_{2k-3})=(x_0+x_2)+(x_0+x_4)+\cdots+(x_0+x_{2k-2})+u,\]
which provides a good $k$-energy tuple. Notice that $x_0+x_i\ne x_0+x_j$ for all pairs $(i, j)$, and $u\ne x_0, x_0+x_i$ for all $i$. Since the additive energy is invariant under the action of orthogonal matrices,  we have $N(X)$ good $k$-energy tuples in $E$. Summing over all $X$, we have $N(Q)$ good $k$-energy tuples in $E$.
\end{proof}
In the form of Lemma \ref{step1},  in order to complete the proof of Theorem \ref{kenergy}, we have to find a lower bound for $N(Q)$, which will be followed by a lower bound of $|Q|$ and Theorem \ref{lyalll}. The following proposition plays an important role for this step.
\begin{proposition}\label{confi} For $d\ge \max\{ 2k-2, 4\}$ and $k\ge 2$, we have $|Q|\gg q^{2k^2-3k}$. 
\end{proposition}
With Proposition \ref{confi} in hand, we derive the following corollary. 
\begin{corollary}\label{qq}
Let $d, k\in \mathbb{N}$ with $d\ge 4k+2$, $\alpha\in (0, 1)$ and $q\ge q(\alpha, k)$. Let $E\subset S^{d-1}\subset \mathbb{F}_q^d$ with $|E|\ge \alpha q^{d-1}$. We have 
\[N(Q)\ge c(\alpha, k)q^{\frac{(2k-1)d-4k}{2}}.\]
\end{corollary}
\begin{proof}
For each configuration in $Q$,  we know from Theorem \ref{lyalll} that the number of  its copies in $E$ is at least 
\[c(\alpha, k)q^{\frac{(2k-1)d-(2k-1)(2k)}{2}}.\]
Taking the sum over all possible $q^{2k^2-3k}$ congruence classes, the lemma follows. 
\end{proof}
Combining Lemma \ref{step1} and Corollary \ref{qq}, Theorem \ref{kenergy} is proved. 
\subsection{Proof of Proposition \ref{confi}}
We now turn our attention to the Proposition \ref{confi}. The proof of Proposition \ref{confi} is quite complicated, which combines the usual Cauchy-Schwarz argument and the claim that most $k$-energy tuples in $S^{d-1}$ are $2k$-spherical configurations spanning $(2k-2)$ dimensions. We first start with some technical lemmas.

\begin{lemma}[Lemma 4.5, \cite{dothang}]\label{do}
For any $E\subseteq S^{d-1}$, and $k\ge 2$,  we have 
\[\left\vert T_k(E)-\frac{|E|^{2k-1}}{q} \right\vert \le q^{\frac{d-1}{2}} T_k^{1/2}T_{k-1}^{1/2},\]
where $T_1(E)=|E|$. 
\end{lemma}
\begin{corollary}\label{co11}
For $k, d\ge 2$, we have
\[T_k(S^{d-1})=(1+o(1))\frac{|S^{d-1}|^{2k-1}}{q}.\]
\end{corollary}
\begin{proof}
We prove by induction on $k$. 

For $k=2$, we apply Lemma \ref{do} to obtain 
\[\left\vert T_2(S^{d-1})-\frac{|S^{d-1}|^3}{q}\right\vert \le q^{\frac{d-1}{2}}\cdot T_2^{1/2} |S^{d-1}|^{1/2}.\]
Using the fact that $|S^{d-1}|\sim q^{d-1}$ and set $x=\sqrt{T_2(S^{d-1})}$, we have 
\[x^2\ge c_1q^{3d-4}-c_2q^{d-1}x, ~\mathtt{and}~x^2\le c_1q^{3d-4}+c_2q^{d-1}x,\]
for some positive constants $c_1$ and $c_2$. 
Solving these equations gives us $x\gg q^{\frac{3d-4}{2}}$ and $x\ll q^{\frac{3d-4}{2}}$, respectively. Thus, the base case is proved. 

Suppose that the claim holds for any $k-1\ge 2$, we now show that it also holds for the case $k$. Indeed, set $x=\sqrt{T_k(S^{d-1})}$, applying Lemma \ref{do} and the inductive hypothesis, we have 
\[x^2-q^{\frac{d-1}{2}}|S^{d-1}|^{\frac{2k-3}{2}} x-\frac{|S^{d-1}|^{2k-1}}{q}\le 0, ~ x^2+q^{\frac{d-1}{2}}|S^{d-1}|^{\frac{2k-3}{2}} x-\frac{|S^{d-1}|^{2k-1}}{q}\ge 0.\]
Solving these inequalities will give us 
\[x=(1+o(1))\left(\frac{|S^{d-1}|^{2k-1}}{q} \right)^{1/2}.\]
This completes the proof of the corollary.
\end{proof}
\begin{lemma}\label{adidaphat} For $d>n\ge 2$, let $L$ be the  number of tuples $(v_0, \ldots, v_n)\in (S^{d-1})^{n+1}$ such that $v_i-v_0\in \{a_1(v_1-v_0)+\cdots+a_{i-1}(v_{i-1}-v_0)+a_{i+1}(v_{i+1}-v_0)+\cdots+a_n(v_n-v_0)\colon a_1, \ldots, a_{i-1}, a_{i+1}, \ldots, a_n\ne 0\}$ for some $1\le i\le n$. We have  $L\ll \frac{|S^{d-1}|^{n+1}}{q^{2}}$.
\end{lemma}
\begin{proof}
Without loss of generality,  we count the number of such tuples with $i=n$. 

Let $\chi$ be the principle additive characteristic of $\mathbb{F}_q$. Using the orthogonality of $\chi$, one has
\begin{align*}
L&\le \frac{1}{q^d}\sum_{s\in \mathbb{F}_q^d}\sum_{v_0, \ldots, v_n\in S^{d-1}}\sum_{a_1, \ldots, a_{n-1}\in \mathbb{F}_q^*}\chi\left(s\cdot \bigg((v_n-v_0)-a_1(v_1-v_0)-\cdots-a_{n-1}(v_{n-1}-v_0)\bigg)\right)\\
&=\frac{|S^{d-1}|^{n+1}}{q^{d-n+1}}+\frac{1}{q^d}\sum_{s\ne \textbf{0}}\sum_{v_0, \ldots, v_n\in S^{d-1}}\sum_{a_1, \ldots, a_{n-1}\in \mathbb{F}_q^*}\chi\left(s\cdot \bigg((v_n-v_0)-a_1(v_1-v_0)-\cdots-a_{n-1}(v_{n-1}-v_0)\bigg)\right)\\
&=\frac{|S^{d-1}|^{n+1}}{q^{d-n+1}}+\frac{1}{q^d}\sum_{s\ne 0}\sum_{a_1, \ldots, a_{n-1}\in \mathbb{F}_q^*}\widehat{S^{d-1}}(a_1s)\cdots \widehat{S^{d-1}}(a_{n-1}s)\widehat{S^{d-1}}(s)\widehat{S^{d-1}}(s(1-a_1-\cdots-a_{n-1})),
\end{align*}
where $\widehat{S}(m)=\sum_{x\in \mathbb{F}_q^d}S(x)\chi(-x\cdot m)$. We now recall from \cite[Lemma 5.1]{ir} that $|\widehat{S^{d-1}}(m)|\ll q^{\frac{d-1}{2}}$ for $m\ne 0$ and $\widehat{S}(\mathbf{0})=|S^{d-1}|\sim q^{d-1}$. 
We now partition the sum $\sum_{a_1, \ldots, a_{n-1}\in \mathbb{F}_q^*}$ into two sub-summands $\sum_{a_1+\cdots+a_{n-1}\ne 1}$ and $\sum_{a_1+\cdots+a_{n-1}=1}$.

Therefore,
\begin{align*}
\sum_{a_1+\cdots+a_{n-1}\ne 1}\widehat{S^{d-1}}(a_1s)\cdots \widehat{S^{d-1}}(a_{n-1}s)\widehat{S^{d-1}}(s)\widehat{S^{d-1}}(s(1-a_1-\cdots-a_{n-1}))\ll q^{\frac{(d-1)(n+1)}{2}}\cdot q^{n-1},
\end{align*}
and 
\begin{align*}
\sum_{a_1+\cdots+a_{n-1}= 1}\widehat{S^{d-1}}(a_1s)\cdots \widehat{S^{d-1}}(a_{n-1}s)\widehat{S^{d-1}}(s)\widehat{S^{d-1}}(s(1-a_1-\cdots-a_{n-1}))\ll q^{\frac{(d-1)(n)}{2}}\cdot q^{d-1}\cdot q^{n-2}.
\end{align*}
These upper bounds are at most $\frac{|S^{d-1}|^{n+1}}{q^{d-n+1}}$ when $d>n$ and $n\ge 2$. In other words,  
\[L\ll \frac{|S^{d-1}|^{n+1}}{q^2}.\]
\end{proof}
\begin{lemma}\label{xyz}
Suppose that $d>2k-2$ and $k\ge 2$. The number of tuples $\{x_0, x_0+x_1, \ldots, x_0+x_{2k-2}, x_0+\sum_{i=1}^{2k-2}(-1)^{i+1}(x_0+x_i)\}$ in $(S^{d-1})^{2k}$ such that 
\[x_0+(x_0+x_1)+(x_0+x_3)+\cdots+(x_0+x_{2k-3})=(x_0+x_2)+(x_0+x_4)+\cdots+(x_0+x_{2k-2})+u,\]
where $u=x_0+\sum_{i=1}^{2k-2}(-1)^{i+1}(x_0+x_i)$, and $x_{i}\in \mathtt{Span}(x_1, \ldots, x_{i-1}, x_{i+1}, \ldots,  x_{2k-2})$ for some $1\le i\le 2k-2$ is $o(T_k(S^{d-1}))$ .
\end{lemma}
\begin{proof}
Applying Lemma \ref{adidaphat} for the family of vectors $\{x_0, x_0+x_1, \ldots, x_0+x_{2k-2}\}$ or its sub-families, we know that  there are at most $\frac{|S^{d-1}|^{2k-1}}{q^2}$ such tuples whenever $d>2k-2$ and $k\ge 2$. We also know from Corollary \ref{co11} that $T_k(S^{d-1})=(1+o(1))\frac{|S^{d-1}|^{2k-1}}{q}$. Thus, the lemma follows from the fact that 
\[\frac{|S^{d-1}|^{2k-1}}{q^2}=o\left(\frac{|S^{d-1}|^{2k-1}}{q}\right).\]
\end{proof}
We are ready to give a proof of Proposition \ref{confi}. 
\begin{proof}[Proof of Proposition \ref{confi}]
For any $k$-energy tuple $(a_1, \ldots, a_k, b_1, \ldots, b_k)\in S^{d-1}$, i.e. 
\begin{equation}\label{xk}a_1+\cdots +a_k=b_1+\cdots+b_k,\end{equation}
we set $a_i=a_1+x_i$ for $2\le i\le k$, and $ b_i=a_1+y_i$ for $1\le i\le k$.  

We first show that most of all tuples $(a_1, \ldots, a_k, b_1, \ldots, b_k)$ satisfying (\ref{xk}) will have the following properties 
\begin{itemize}
\item[a.] $\{x_2, \ldots, x_k, y_1, \ldots, y_k\}$ are linearly independent.
\item[b.] $||x_i-x_j||\ne 0, ||y_i-y_j||\ne 0$ for all pairs $i\ne j$, and  $||x_i-y_j||\ne 0, ||x_i||\ne 0, ||y_j||\ne 0$ for all pairs $i, j$.
\item[c.] For any $I, J\subseteq \{1, \ldots, k\}$, we have $\sum_{i\in I}a_i-\sum_{j\in J}b_j\ne 0$.
\end{itemize}
Indeed, let $T_k^{\mathtt{dep}}(S^{d-1}), T_k^{\mathtt{0}}(S^{d-1}), T_k^{\mathtt{bad}}(S^{d-1})$ be the number of $k$-energy tuples not satisfying (a), (b), and (c), respectively. We will prove that $T_k^{\mathtt{dep}}(S^{d-1}), T_k^{\mathtt{0}}(S^{d-1}), T_k^{\mathtt{bad}}(S^{d-1})=o(T_k(S^{d-1}))$.

{\bf Bounding $T_k^{\mathtt{dep}}$:} By Lemma \ref{xyz}, we have $T_k^{\mathtt{dep}}=o(T_k(S^{d-1}))$.

{\bf Bounding $T_k^0$:} It follows from our setting that $||x_i-x_j||=||a_i-a_j||$ and $||x_i-y_j||=||a_i-b_j||$. Hence, it is sufficient to count tuples with $||a_i-a_j||=0$ for some $1\le i\ne j\le k$. The other cases can be treated in the same way. 

Without loss of generality, we assume that $||a_1-a_2||=0$, which is equivalent with $||x_2||=0$. 

Let $U$ be the multi-set defined by 
\[U:=\{a_1+\cdots+a_k\colon a_i\in S^{d-1}, ||a_1-a_2||=0 \}.\]
Let $W$ be the multi-set defined by 
\[W:=\{b_1+\cdots+b_{k-1}\colon b_i\in S^{d-1}\}.\]

Let $e(U, W)$ be the number of pairs $(u, w)\in U\times W$ such that $u-w\in S^{d-1}$. Applying Lemma \ref{exp} for the graph $C_{\mathbb{F}_q^d}(S^{d-1})$, we have
\[e(U, W)\le \frac{|U||W|}{q}+q^{\frac{d-1}{2}}\left(\sum_{u\in \overline{U}}m(u)^2\right)^{1/2}\cdot \left(\sum_{w\in \overline{W}}m(w)^2\right)^{1/2},\]
where $m(u), m(w)$ are the multiplicities of $u$ and $w$ in $U$ and $W$, respectively. 

We know from \cite{hart} that for any two sets $X, Y\subseteq S^{d-1}$, the number of pairs $(x, y)\in X\times Y$ such that $||x-y||=0$ is at most $\frac{|X||Y|}{q}+q^{\frac{d}{2}}|X|^{1/2}|Y|^{1/2}$. 
So with $X=Y=S^{d-1}$, we obtain $|U|\le \frac{|S^{d-1}|^k}{q}$. It is clear that  $|W|=|S^{d-1}|^{k-1}$.

On the other hand, it is not hard to see that 
\[\sum_{u}m(u)^2\le T_k(S^{d-1}), ~\sum_{w}m(w)^2\le T_{k-1}(S^{d-1}).\]

Using Corollary \ref{co11}, one has 

\[e(U, W)\le \frac{|S^{d-1}|^{2k-1}}{q^2}+q^{\frac{d-1}{2}}\cdot \frac{|S^{d-1}|^{\frac{2k-1}{2}}}{q^{1/2}}\cdot \frac{|S^{d-1}|^{\frac{2k-3}{2}}}{q^{1/2}}\ll  \frac{|S^{d-1}|^{2k-1}}{q^2}.\]
On the other hand, $e(U, W)$ equals to the number of tuples satisfying (\ref{xk}) with $||a_1-a_2||=0$. 

In other words, 
\[T_k^0\ll  \frac{|S^{d-1}|^{2k-1}}{q^2}=o(T_k(S^{d-1})).\]
{\bf Bounding $T_k^{\mathtt{bad}}$:}
Let $I$ and $J$ be two subsets of $\{1, \ldots, k\}$. Assume that $|I|=|J|=m$. The case $|I|\ne |J|$ is treated in the same way. Without loss of generality, we assume that $I=J=\{1, \ldots, m\}$. We now count the number of $k$-energy tuples $(a_1, \ldots, a_k, b_1, \ldots, b_k)\in (S^{d-1})^{2k}$ such that $a_1+\cdots+a_m-b_1-\cdots-b_m=0$. This implies that $a_{m+1}+\cdots+a_k-b_{m+1}-\cdots-b_k=0$. 

We now show that the number of tuples $(a_1, \ldots, a_m,b_1, \ldots, b_m)\in \left(S^{d-1}\right)^{2m}$ such that $a_1+\cdots+a_m-b_1-\cdots-b_m=0$ is at most $\ll \frac{|S^{d-1}|^{2m-1}}{q}$. 

Indeed, using the same argument as in bounding $T_k^0$, let $U', W'$ be multi-sets defined by 
\[U':=\{a_1+\cdots+a_m\colon a_i\in S^{d-1}\},~~W=\{b_1+\cdots+b_{m-1}\colon b_i\in S^{d-1}\}.\]
The number of such tuples is bounded by $e(U', W')$ in the graph $C_{\mathbb{F}_q^d}(S^{d-1})$. As before, we also have 
\[\sum_{u\in \overline{U'}}m(u)^2=T_m(S^{d-1}), ~\sum_{w\in \overline{W'}}m(w)=T_{m-1}(S^{d-1}).\]
Using Lemma \ref{exp} and Lemma \ref{do}, we have 
\[e(U', W')\ll \frac{|S^{d-1}|^{2m-1}}{q}+q^{\frac{d-1}{2}}\cdot \frac{|S^{d-1}|^{2m-2}}{q}\ll \frac{|S^{d-1}|^{2m-1}}{q}.\]
Similarly, the number of tuples $(a_{m+1}, \ldots, a_k, b_{m+1}, \ldots, b_k)\in S^{d-1}$ such that $a_{m+1}+\cdots+a_k-b_{m+1}-\cdots-b_k=0$ is at most $\ll \frac{|S^{d-1}|^{2(k-m)-1}}{q}$.

Hence, the number of $k$-energy tuples with $\sum_{i\i I}a_i-\sum_{j\in J}b_j=0$ is at most $\ll \frac{|S^{d-1}|^{2k-2}}{q^2}$.  

Summing over all possibilities of sets $I$ and $J$, we obtain 
\[T_k^{\mathtt{bad}}(S^{d-1})=o(T_k(S^{d-1}).\]
From the bounds of $T_k^{\mathtt{dep}}$, $T_k^0$, and $T_k^{\mathtt{bad}}(S^{d-1})$, we conclude that most of $k$-energy tuples in $S^{d-1}$ satisfying $(a), (b)$, and $(c)$. We denote the number of those tuples by $T_k^*(S^{d-1})$. 

We recall that for any two non-trivial spherical configurations $X$ and $X'$, they are in the same congruent class if there exists $g\in O(d, \mathbb{F}_q)$ such that $gX=X'$. For each configuration in $Q$, say, \[X=\{x_0, x_0+x_1, x_0+x_2, \ldots, x_0+x_{2k-2}, x_0+\sum_{i=1}^{2k-2}(-1)^{i+1}(x+x_i)\},\]
the $2k-1$ vertices $x_0, x_0+x_1, \ldots, x_{0}+x_{2k-2}$ form a non-degenerate $(2k-2)$-simplex. We know from \cite{bennet} that the stabilizer of a non-degenerate $(2k-2)$--simplex in $S^{d-1}$ is of cardinality at least $|O(d-2k+1)|$. 

For any $X\in Q$, let $\mu(X)$ be the number of configurations which are congruent to $X$. We have $\sum_{X\in Q}\mu(X)=T_k^*(S^{d-1})$. By Cauchy-Schwarz inequality, we have 
\begin{equation}\label{eq134}\sum_{X\in Q}\mu(X)\le |Q|^{1/2}\cdot \left(\sum_{X}\mu(X)^2\right)^{1/2}.\end{equation}
On the other hand, $\sum_{X}s(X)\mu(X)^2$ is at most the number of pairs of configurations $(X, X')$ such that $X'=g(X)$ for some $g\in (d, \mathbb{F}_q)$, where $s(X)$ is the stabilizer of $X$.
Hence, we can bound $\sum_{X}s(X)\mu(X)^2$ by $T_k^*(S^{d-1})\cdot |O(d, \mathbb{F}_q)|$. This implies that 
\begin{equation}\label{eq234}\sum_{X}\mu(X)^2\le \frac{|O(d, \mathbb{F}_q)|\cdot T_k^*(S^{d-1})}{|O(d-2k+1)|}.\end{equation}
We recall from \cite{bennet} that $|O(n, \mathbb{F}_q)|\sim q^{\binom{n}{2}}$. From (\ref{eq134}) and (\ref{eq234}), we obtain  $|Q|\gg q^{2k^2-3k}$. This completes the proof.
\end{proof}
\section{Proof of Proposition \ref{thm1.9}}
\begin{proof}[Proof of Proposition \ref{thm1.9}]
Suppose $q=p^r$ with $r= \frac{1}{\epsilon}$ (assume that $1/\epsilon$ is an integer).

Let $\mathcal{A}$ be an arithmetic progression in $\mathbb{F}_q$ of size $p^{r-1}$. Let $X$ be the hyperplane $x_d=0$. Define 
\[H:=\{X+(0, \ldots, 0, a)\colon a\in \mathcal{A}\}.\]
Note that $H$ is a set of $|\mathcal{A}|$ translates of the hyperplane $X$. 

We have $|H|=q^{d-1}\cdot q^{\frac{r-1}{r}}= q^{d-\epsilon}$. It is not hard to see that 
\[|(H-H)\cap S^{d-1}|\ll q^{d-2}\cdot q^{\frac{r-1}{r}}\ll q^{d-1-\epsilon}=o(|S^{d-1}|).\]
For any $1\le m\ll |S^{d-1}|$, let $E\subset S^{d-1}\setminus (H- H)$ with $|E|=m$, we have \begin{equation}\label{contradic}(H-H)\cap E=\emptyset.\end{equation}

If $\mu<\frac{|E|}{2q^{\epsilon}}$, then by Lemma \ref{exp} for the graph $C_{\mathbb{F}_q^d}(E)$, one has 
\[e(H, H)\ge \frac{|H|^2|E|}{q^d}-\frac{|E||H|}{2q^{\epsilon}}>0,\]
whenever $|H|>\frac{q^{d-\epsilon}}{2}$, which contradicts to (\ref{contradic}).

In other words, we have $\mu \ge \frac{|E|}{2q^{\epsilon}}$. 

In the case $|E+E|=K|E|<q^d/2$, we start with an observation that 
\[T_2(E)\ge \frac{|E|^4}{|E+E|},\]
which implies
\[T_2(E)\ge \frac{|E|^4}{K|E|}. \]
Let $X$ be the multi-set in $\mathbb{F}_q^d$ defined by $X=E+E$. We can apply the Expander mixing lemma for the graph $C_{\mathbb{F}_q^d}(E)$ to get an upper bound for $T_2(E)$. Indeed, one has 
\[T_2(E)=e(X, -E)\le \frac{|E|^4}{q^d}+\mu \cdot T_2(E)^{1/2}\cdot |E|^{1/2}.\]
This gives us 
\[T_2(E)\le \frac{|E|^4}{q^d}+\mu^2\cdot |E|.\]
Since $K|E|<q^d/2$, we have $\mu^2|E|\gg \frac{|E|^4}{K|E|}$. This gives $\mu\gg \frac{|E|}{K^{1/2}}$. Hence, when $K\sim 1$, we have $\mu\gg |E|$. 
\end{proof}
\section{Proof of Proposition \ref{pro14}}
\begin{proof}[Proof of Proposition \ref{pro14}]
We have seen in the proof of Theorem \ref{thm-lowerbound} that the number of cycles of length $2k$ is equal to $q^d\cdot T_k^*(S^{d-1})$ and $T_k^*(S^{d-1})=(1+o(1))|S^{d-1}|^{2k-1}/q$. Hence, the number of cycles of length $2k$ in $C_{\mathbb{F}_q^d}(S^{d-1})$ is $(1+o(1))|S^{d-1}|^{2k-1}q^{d-1}$.

To prove the upper bound on the number of cycles of length $2k$ in a given set $A\subset \mathbb{F}_q^d$, we need to recall the following result from \cite{bene}. 
\begin{theorem}[Bennett-Chapman-Covert-Hart-Iosevich-Pakianathan, \cite{bene}]\label{Mophat-Adidaphat} For $A\subset \mathbb{F}_q^d$, $d\ge 2$ and an integer $k\ge 1$. Suppose that $\frac{2k}{\ln 2}q^{\frac{d+1}{2}}=o(|A|)$ then the number of paths of length $k$ with vertices in $A$ in $C_{\mathbb{F}_q^d}(S^{d-1})$ is $(1+o(1))\frac{|A|^{k+1}}{q^k}$.
\end{theorem}

Let $N$ be the number of cycles of length $2k$ with vertices in $A$. For any two vertices $x, y\in A$, let $P(x, y)$ be the number of paths of length $k$ between $x$ and $y$ with vertices in $A$. It follows from Theorem \ref{Mophat-Adidaphat} that 
\[\sum_{x, y\in A}P(x, y)=(1+o(1))\frac{|A|^{k+1}}{q^k}.\]

It is clear that 
\[N=\sum_{x, y\in A}\binom{P(x, y)}{2}.\]
Using the convexity of the function $\binom{x}{2}$, one has 
\[N\gg |A|^2\cdot \binom{\frac{\sum_{x, y\in A}P(x, y)}{|A|^2}}{2}\gg \frac{|A|^{2k}}{q^{2k}},\]
provided that $|A|\gg q^{\frac{k}{k-1}}$. This completes the proof. 
\end{proof}
\begin{remark}\label{rmm}
We remark here that for any $k\ge 2$, there exists a set $E\subseteq S^{d-1}$ with $|E|\gg q^{\frac{d}{2k-1}}$ such that all cycles of length $2k$ in $C_{\mathbb{F}_q^d}(E)$ do not have distinct vertices. Such a set can be constructed easily as follows. Let $H$ be a $2k$-uniform hypergraph with the vertex set $S^{d-1}$, and each edge is a good $k$-energy tuple, then we know from the proof of Proposition \ref{confi} that the number of edges in $H$ is at most $|S^{d-1}|^{2k-1}/q$. Applying  Spencer's independent hypergraph number lemma  in \cite{spencer}, we get an independent set $E$ of size at least $\gg q^{\frac{d}{2k-1}}$. This set will satisfy our desired properties. 
\end{remark}

\section*{Acknowledgments}
The author was supported by Swiss National Science Foundation grant P4P4P2-191067. I would like to thank  Ilya Shkredov for useful discussions about the second eigenvalue of the local Cayley distance graphs.

 \bibliographystyle{amsplain}

\end{document}